\documentclass[12pt]{article}
\usepackage{color}
\usepackage{amsmath}
\usepackage[marginal]{footmisc}
\usepackage{extarrows}
\usepackage{geometry}
\usepackage{authblk} 
\usepackage{hyperref} 
\usepackage{color}
\usepackage{ntheorem}

\usepackage[mathscr]{euscript}
\usepackage[utf8]{inputenc}
\def\0{\emptyset}

\newcommand{\lc}{\left\lceil}
\newcommand{\rc}{\right\rceil}
\newtheorem{theorem}{Theorem}[section]

\newtheorem{lemma}[theorem]{Lemma}

\newtheorem{cor}[theorem]{Corollary}

\newtheorem{construction}[theorem]{Construction}

\newenvironment{reproof}{{\noindent\it Proof of Theorem \ref{thm: ex for k_l,t}.}}{\hfill $\square$\par}

\newenvironment{proof}{{\noindent\it Proof.}}{\hfill $\square$\par}
\newenvironment{reproof1_4}{{\noindent\it Proof of Theorem \ref{thm: ex for k_2,2 free}.}}{\hfill $\square$\par}
\newenvironment{reproof1_5}{{\noindent\it Proof of Theorem \ref{thm: ex for k_2,t free}.}}{\hfill $\square$\par}

\usepackage{graphicx}

\usepackage{enumerate}
\usepackage{enumitem}

\usepackage{
	amsmath,			
	amssymb,			
	enumerate,		    
	graphicx,			
	lastpage,			
	multicol,			
	multirow,			
	pifont,			    
}

\usepackage[numbers]{natbib}

\newcommand{\excm}{ex(n,\{K_{2,2}, M_{s+1}\})}

\newcommand{\exkm}{ex(n,\{K_{l,t}, M_{s+1}\})}

\begin{document}


\title{ Tur\'an number of complete bipartite graphs with bounded matching number}
\author{
    {\small\bf Huan Luo}\thanks{email:  h-luo20@mails.tsinghua.edu.cn}\quad
    {\small\bf Xiamiao Zhao}\thanks{email:  zxm23@mails.tsinghua.edu.cn}\quad
    {\small\bf Mei Lu}\thanks{Corresponding author : email: lumei@tsinghua.edu.cn\\Huan Luo and Xiamiao Zhao should be considered joint first author}\\
    {\small Department of Mathematical Sciences, Tsinghua University, Beijing 100084, China.}\\
    
}
\date{}

\maketitle\baselineskip 16.3pt
\begin{abstract}
Let $\mathscr{F}$ be a family of graphs. A graph $G$ is $\mathscr{F}$-free if  $G$ does not contain any $F\in \mathcal{F}$  as a subgraph. The Tur\'an number $ex(n, \mathscr{F})$ is the maximum number of edges in an $n$-vertex $\mathscr{F}$-free graph. Let $M_{s}$ be the matching
 consisting of $ s $ independent edges.
	Recently, Alon and Frank determined  the exact value of $ex(n,\{K_{m},M_{s+1}\})$.  Gerbner obtained several results about $ex(n,\{F,M_{s+1}\})$ when $F$ satisfies certain proportions. In this paper, we determine the exact value of  $ex(n,\{K_{l,t},M_{s+1}\})$ when $s, n$ are large enough for every $3\leq l\leq t$. When $n$ is large enough, we also show that $\excm=n+{s\choose 2}-\lc\frac{s}{2}\rc$ for $s\ge 12$ and $ex(n,\{K_{2,t},M_{s+1}\})=n+(t-1){s\choose 2}-\lc\frac{s}{2}\rc $ when $t\ge 3$ and $s$ is large enough.
 \end{abstract}


\vskip.3cm

\section{Introduction}

In this paper, we only consider finite, simple and undirected graphs. Let $G=(V,E)$ be a graph, where $V$ is the vertex set and $E$ is the edge set of $G$. We use $n$ and $e(G)$ to denote the order and size of $G$, respectively.

Let $\mathscr{F}$ be a family of graphs. A graph $G$ is $\mathscr{F}$-free if  $G$ does not contain any $F\in \mathcal{F}$  as a subgraph. The Tur\'an number, denoted by $ex(n, \mathscr{F})$, is the maximum number of edges in an $n$-vertex $\mathscr{F}$-free graph. If $\mathscr{F}=\{F\}$, we write $ex(n, F)$ instead of  $ex(n, \mathscr{F})$.

The well-known result  by Tur\'an \cite{turan1941external} asserted that $ex(n,K_{k+1})=e(T(n,k))$, where $T(n,k)$ is the balanced complete $k$-partite graph. The Tur\'an Theorem is considered as the origin of extremal graph theory.

Let $K_{l,t}$ ($l\leq t$) be a complete bipartite graph, the numbers of vertices in each part are $l$ and $t$ respectively. A matching $M_s$ is consist of $s$ independent edges. There are lots of results about $ex(n,K_{l,t})$.   $G$ is $K_{1,t}$-free if and only if the maximum degree of $G$ is less than $t$ and then $ex(n,K_{1,t})=\lfloor\frac{(t-1)n}{2}\rfloor$.
Erd\H{o}s and Gallai \cite{erdos1961minimal} showed that
   $$ ex(n,M_{s+1})=\max\left\{ns-{s+1\choose 2},{2s+1\choose 2}\right\}.$$
For $K_{2,2}=C_4$, i.e. a cycle with length $4$, the first result about $ex(n,C_4)$ was given by K\H{o}v\'ari, S\'os and Tur\'an \cite{kHovari1954problem}. That is
\begin{equation}\label{eq: ex(n,c4)<}
   ex(n,C_4)\leq \frac n4(1+\sqrt{4n-3}).
\end{equation}
Brown \cite{brown1966graphs}, and Erd\H{o}s, R\'enyi and S\'os \cite{erdos1966} independently and simultaneously proved that
$$ex(q^2+q+1,C_4)\geq\frac{1}{2}q(q+1)^2$$
for all prime powers $q$.
F\"uredi \cite{furedi1983graphs,furedi1996number} proved that for all prime powers $q\geq 14,$
$$ex(q^2+q+1,C_4)=\frac{1}{2}q(q+1)^2.$$
He, Ma and Yang \cite{he2021some} showed that when $q$ is large enough and is an even integer, all $C_4$-free graphs $G$ with $q^2+q+1$ vertices and $e(G)\geq \frac{1}{2}q(q+1)^2-0.2q+1$ are subgraphs of a unique polarity graph of order $q.$ But there is still no results about the value of $ex(n,C_4)$ for a general positive integer $n$.
A result of $ex(n,K_{2,t})$ was given by F\"uredi \cite{uredi1996new} who proved when $t\geq 2$,  $$(n,K_{2,t})=(1+o(1))\sqrt{t-1}n^{3/2}/2.$$
On general case when $2\leq l\leq t$, the celebrated K\H{o}v\'ari–S\'os–Tur\'an Theorem \cite{kHovari1954problem} stated that $$ex(n,K_{l,t})= O(n^{2-\frac{1}{t}}),$$ one of its proof was given by Alon, Krivelevich and Sudakov \cite{alon2003turan}. But the upper bound has not been proved to be tight unless $l\in \{2,3\}$ or $t$ is sufficiently large with respect to $l$ \cite{alon1999norm,bukh2021extremal}.

Recently, Alon and Frankl \cite{alon2024turan} proposed to consider the maximum number of edges of a $F$-free graph on $n$ vertices with matching number at most $s$, that is, $ex({n,\{F,M_{s+1}\}})$. In the same paper, they proved that $ex(n,\{K_{k+1},M_{s+1}\})=e(G(n,k,s)),$
where $G(n,k,s)$ denotes the complete $k$-partite graph on $n$ vertices consisting of $k-1 $ vertex classes of sizes as equal as possible whose total size is $s$, and one additional vertex class of size $n-s$. They also showed that for any $F$ with chromatic number $k+1$ and a color-critical edge (an edge that will decrease the chromatic number when deleted), $ex(n,\{F,M_{s+1}\})=e(G(n,k,s)).$
Later, Gerbner \cite{gerbner2024turan} gave several results about $ex({n,\{F,M_{s+1}\}})$. For example, he proved that if the chromatic number of $F$ is at least 3 and $n$ is large enough, then $ex(n,\{F,M_{s+1}\})=ex(s,\mathscr{F})+s(n-2)$, where $\mathscr{F}$ is the family of graphs obtained by deleting an independent set from $F$. If $F$ is a bipartite graph, let $p(F)$ denote the smallest possible order of a color class in a proper two-coloring of $F$, then $ex(n,\{F,M_{s+1}\})=(p-1)n+O(1)$, where $p(F)\leq s$.
Chv\'atal and Hanson \cite{chvatal1976degrees} gave the result of maximum star-free graph with bounded matching number, i.e. $ex(n,\{K_{1,t},M_{s+1}\})$. See \cite{chvatal1976degrees} for more details.

In this paper, we consider $ex(n,\{K_{l,t},M_{s+1}\})$ for $t\ge l\geq 2$. Note that if $s\leq l$, then $ex(n,\{K_{l,t},M_{s+1}\})= ex(n,M_{s+1})$ because the matching number of $K_{l,t}$ is $l$. On the other hand, if $n\le 2s$, then $ex(n,\{F,M_{s+1}\})=ex(n,F).$ Hence we just consider the case $s\ge l+1$ and $n\ge 2s+1$. Let
$$\begin{array}{rcl}
F_1(x)&=&(l-1)n+(t-1){x\choose l}+{2(s-x)+1\choose 2}-\lc\frac{x(l-1)}{2}\rc,\\
   F_2(x) &=& (l-1)n+(t-1){x\choose l}+ex(2(s-x)+1,K_{l,t})-\lc\frac{x(l-1)}{2}\rc-(2(s-x)+1)(l-1).
   \end{array}$$ Denote $r_1=\max_{\max\{2s-t+1,2l\}\leq 2x\leq 2s }\{F_1(x)\}$ and $r_2=\max_{2l\leq 2x\leq 2s-t}\{F_2(x)\}$.
   Let $$n(x) = x+(t-1){x\choose l}+2(s-x)+1$$
   and $n_1 = \max_{0\leq x\leq s}n(x)$. It is easy to find that $n_1 = \max\{n(0),n(s)\}$. 
   The following theorems are our main results.
\begin{theorem}\label{thm: ex for k_l,t}
    Let $3\leq l\leq t$ and $n\geq n_1$.
   Then we have
   $$\exkm\geq \max \{r_1,r_2\}.$$
   When $s$ is large enough and $n\geq 2{3s\choose 2}$, we have
   $$\exkm\leq (l-1)n+(t-1){s\choose l}-\lc\frac{s(l-1)}{2}\rc.$$
\end{theorem}
  If $s$ is large enough, we can show that the exact value of $\exkm$ is the upper bound in Theorem \ref{thm: ex for k_l,t}.
   \begin{theorem}\label{thm: exkm when s is large enough}
       Assume $3\leq l\leq t$. If $s$ is large enough and $n\geq n(s)$, then
       $$\exkm= (l-1)n+(t-1){s\choose l}-\lc\frac{s(l-1)}{2}\rc.$$
   \end{theorem}
 \begin{proof}
 It is easy to check that when $s$ is large enough and $l\geq 3$, $n_1 = n(s)$ and $n(s)\geq 2{3s\choose 2}$.
  Note that $ex(2(s-x)+1,K_{l,t})=O((s-x)^{2-1/t}).$ If  $l\ge 3$, $F_1 (x)$ and $F_2 (x)$ are monotonically increasing when $s$ is large  enough.
  Hence  $$\max \{r_1,r_2\}=F_1(s)=(l-1)n+(t-1){s\choose l}-\lc\frac{s(l-1)}{2}\rc.$$ By Theorem \ref{thm: ex for k_l,t}, we are done.
 \end{proof}

 For $K_{2,t}$ with $t\ge 2$, we can give more specific results.
\begin{theorem}\label{thm: ex for k_2,2 free}
    Let $n\geq {s\choose 2}+s+1$ and $s\geq 12$. Then we have $$\excm=n+{s\choose 2}-\lc\frac{s}{2}\rc.$$
\end{theorem}
\begin{theorem}\label{thm: ex for k_2,t free}
   If $n\geq {s\choose 2}+s+1$, $t\ge 3$ and $s$ is large enough, then we have$$ex(n,\{K_{2,t},M_{s+1}\})=n+(t-1){s\choose 2}-\lc\frac{s}{2}\rc.$$
\end{theorem}

This paper is organized as follows. In section~\ref{sec: preliminary}, we give basic definitions, notations, the construction of extremal graphs and lemmas.
In section \ref{sec: maximum k_lt free}, we will prove Theorem \ref{thm: ex for k_l,t}; In section \ref{sec: ex for k2t free}, we will show Theorems \ref{thm: ex for k_2,2 free} and \ref{thm: ex for k_2,t free}.

\section{Preliminary}\label{sec: preliminary}

Let $G = (V, E)$ be a graph on $n$ vertices. For  $S\subseteq V(G)$, we denote $N_S(v)=\{u\in S|uv\in E(G)\}$ and $d_S(v)=|N_S(v)|$. We use $G[S]$ to denote the subgraph of $G$ induced by $S$ and will used $G-S$ to denote $G[V(G)\setminus S]$.
Let $S_1,S_2\subseteq V(G)$ with $S_1\cap S_2=\emptyset$. Denote $(S_1,S_2)=\{uv\in E(G)|u\in S_1,v\in S_2\}$ and $e(S_1,S_2)=|(S_1,S_2)|$.

Let  $X\subseteq V(G)$ and $T\subseteq X$ with $|T|=l$.
 We call $T$ a {\bf $l$-horn of $X$} if there is  $v\in V(G)\setminus T$ such that $T\subseteq N(v)$ and we call $v$ the center of $T$. Denote $C(T)=\{v~|~v\mbox{~is a center of $T$}\}$. If $G$ is $K_{l,t}$-free, then $|C(T)|\le t-1$ for any
$l$-horn $T$ of $X$.

The following theorem is about the matching number of a graph.
\begin{theorem}[Tutte-Berge Theorem~\cite{lovasz2009matching}]\label{TB}
Let $G$ be a graph. Then the matching number of $G$ is
$$\frac{1}{2}\min_{X\subseteq V(G)}(|V(G)|+|X|-odd(G-X)),$$where $odd(G-X)$ is the number of odd components of $G-X$.
\end{theorem}

By
 Theorem  \ref{TB}, we  have the following result which describes a helpful property of graphs with bounded matching number.
\begin{lemma}\label{thm: TB theorem}
     Let $G$ be a graph with $n$ vertices. Then the matching number of $G$ is at most $s$ if and only if there is a subset $X\subseteq V(G)$ such that
     \begin{equation}\label{eq: TB inequaulity}
         |X|+\sum_{i=1}^m \left\lfloor\frac{|V(C_i)|}{2}\right\rfloor\leq s,
     \end{equation}
      where $C_1,\ldots,C_m$ are  the components of $G-X $.
\end{lemma}
\begin{proof} Assume the matching number of $G$ is at most $s$. By
 Theorem  \ref{TB}, there is $X\subseteq V(G)$ such that $|V(G)|+|X|-odd(G-X)\le 2s$. Let $odd(G-X)=t$ and $C_1,\ldots,C_m$ be  the components of $G-X $. Assume $C_1,\ldots,C_t$ are the odd components of $G-X$. Let $|X|=x$ and $|V(C_i)|=c_i$ for $1\le i\le m$. Then we have $n+x-t\le 2s$ and $x+\sum_{i=1}^mc_i=n$. Then
$$|X|+\sum_{i=1}^m \left\lfloor\frac{|V(C_i)|}{2}\right\rfloor=\frac{n-x}{2}-\frac{t}{2}+x\le s.$$

Assume there is a subset $X\subseteq V(G)$ such that $|X|+\sum_{i=1}^m \left\lfloor\frac{|V(C_i)|}{2}\right\rfloor\leq s$. Then $|V(G)|+|X|-odd(G-X)\le s$. By Theorem  \ref{TB}, the matching number of $G$ is at most $s$.
\end{proof}

Now we describe two constructions based on Lemma \ref{thm: TB theorem}. Let $s,j,t,n$ be integers such that $n\ge 2s+1$, $t\ge l\ge 2$ and $s\ge l+1$.
\begin{construction}\label{con: k_l,t-free 1}
    Let $x,n,l,s$ be  integers satisfying $s\ge x\geq l\geq 2$, $2(s-x)+1\leq t$ and  $n\geq x+(t-1){x\choose l}+2(s-x)+1 $. Let $X$ (resp. $S$) be a $K_{1,l}$-free graph (resp. $K_{l,t}$-free graph) with $|V(X)|=x$ and $e(X)=ex(x,K_{1,l})$ (resp. $|V(S)|=2(s-x)+1$ and $e(S)=ex(2(s-x)+1,K_{l,t})$). Set $\mathscr{F}:=\{F\subset V(X)| |F|=l \}$. For each $F\in \mathscr{F}$, we add a vertex set
    $V_F=\{v_{F_1},\ldots,v_{F_{t-1}}\}$ consisting of
    $t-1$ new isolated vertices and connect each vertex in $F$ with all $v_{F_i}$ for $i=1,\dots,t-1$. Arbitrarily choose  $F_0\subset V(X)$ such that $|F_0|=l-1$. Then we connect all vertices of $S$ to each vertices in $F_0$. Finally, we add a new vertex set $U$ consisting of $n-x-(t-1){x\choose l}-2(s-x)-1$ isolated vertices and
        connect each vertices in $F_0$ with each vertex in $U$. We denote the obtained graph by $G_1^x$.
\end{construction}
Since $2(s-x)+1\leq t$, we have $S$ is a complete graph which implies $e(S)={2(s-x)+1\choose 2}$. From the Construction \ref{con: k_l,t-free 1}, $V(G_1^x)=V(X)\cup (\cup_{F\in \mathscr{F}}V_F)\cup V(S)\cup U$ and
\begin{equation}\label{eq: con1 of K_l,t}
    e(G_1^x)=(l-1)n+(t-1){x\choose l}-\lc\frac{x(l-1)}{2}\rc+{2(s-x)+1\choose 2}=F_1(x).
\end{equation}

\begin{construction}\label{con: k_l,t-free 2}Let $x,n,l,s$ be  integers satisfying $s\ge x\geq l\geq 2$, $2(s-x)+1\geq t+1$ and  $n\geq x+(t-1){x\choose l}+2(s-x)+1 $. We have the same construction as the Construction \ref{con: k_l,t-free 1}, but deleting all the edges between $S$ and $X$. We denote the obtained graph by $G_2^x$.
\end{construction}
Obviously, we have
\begin{equation}\label{eq: con2 of K_l,t}
    e(G_2^x)=(l-1)n+(t-1){x\choose l}-\lc\frac{x(l-1)}{2}\rc+ex(2(s-x)+1,K_{l,t})-(2(s-x)+1)(l-1)=F_2(x).
\end{equation}

\noindent{\bf Remark} In the Construction \ref{con: k_l,t-free 2}, we delete all the edges between $S$ and $X$ to ensure that there are no copy of $K_{l,t}$ in $G_2^x$. In fact,  there are some edges could be added between $S$ and $X$. As a result, the lower bound we propose is not tight. We speculate that when $s$ is small, the extremal graph is difficult to give an accurate description.

\begin{figure}[htbp]
    \begin{minipage}[t]{0.5\linewidth}
        \centering
        \includegraphics[width=0.8\textwidth]{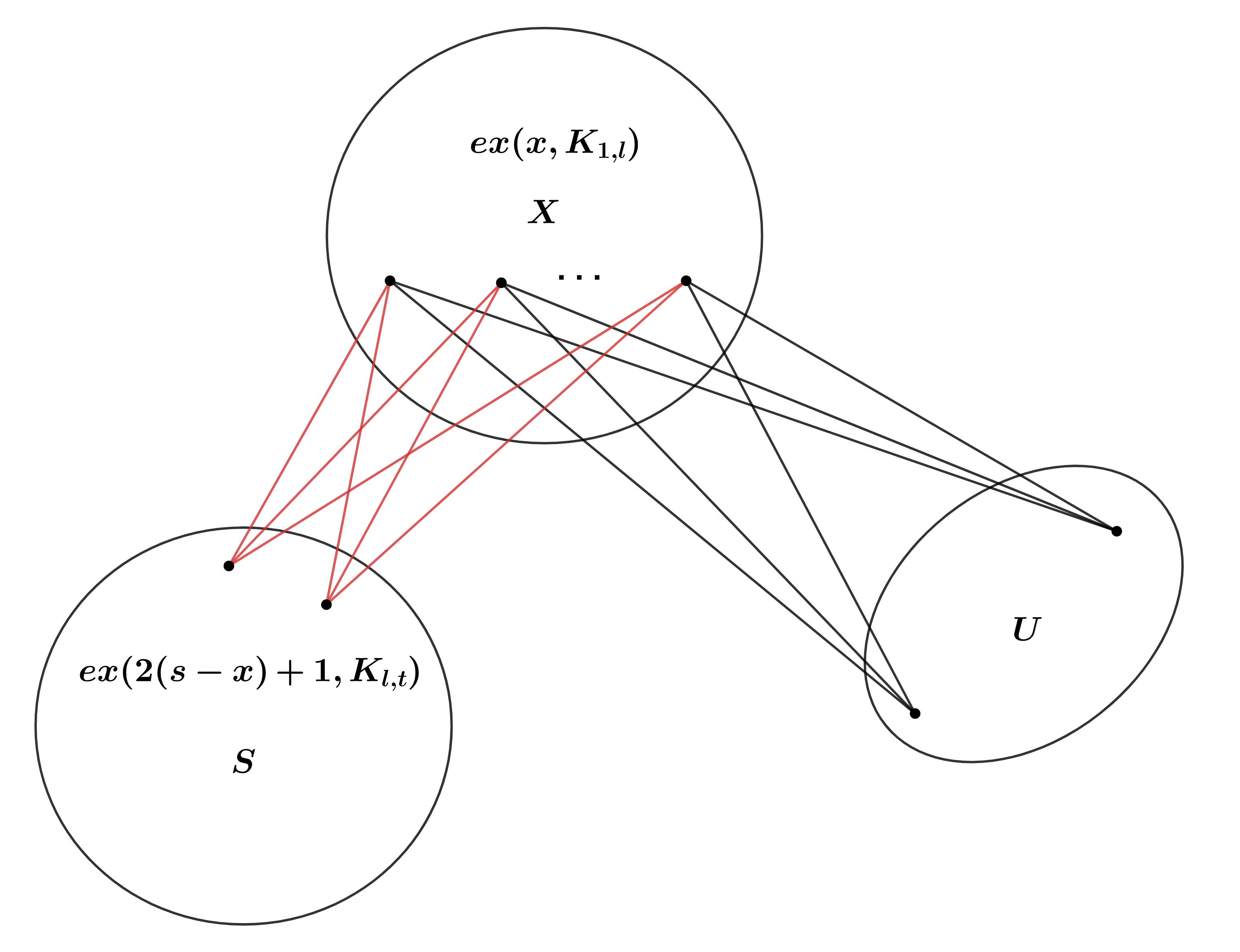}
        \caption{Construction \ref{con: k_l,t-free 1}}
    \end{minipage}%
    \begin{minipage}[t]{0.5\linewidth}
        \centering
        \includegraphics[width=0.8\textwidth]{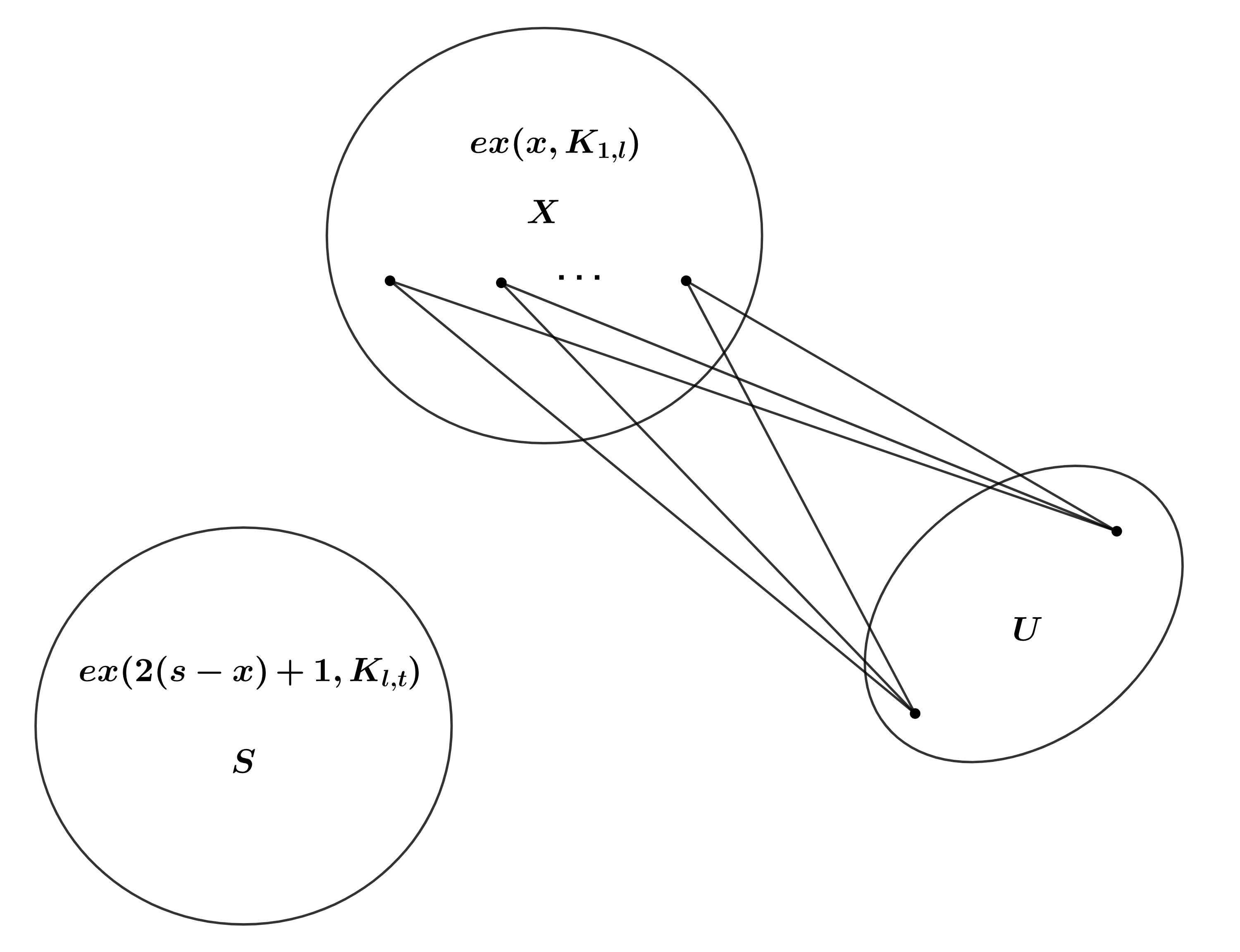}
        \caption{Construction \ref{con: k_l,t-free 2}}
    \end{minipage}
\end{figure}
\begin{lemma}\label{con-lem}
 $G_1^x$ and $G_2^x$ are $\{K_{l,t},M_{s+1}\}$-free.
\end{lemma}
\begin{proof} We just show that $G_1^x$ is $\{K_{l,t},M_{s+1}\}$-free. By the Construction \ref{con: k_l,t-free 1}, $d(u)=l-1$ for any $u\in U$. If these is a $t$-horn $T$ of $V(G_1^x)$ such that $|C(T)|\ge l$, then $T\cap U=\emptyset$.
  If $T\cap (\cup_{F\in\mathscr{F}}V_F)\not=\emptyset$, then $T\cap (V(X)\cup V(S))\not=\emptyset$ by $|V_F|=t-1$, a contradiction with  $X$ being $K_{1,l}$-free and $d_X(v)=l-1$ for each $v\in V(S)$.  Hence $T\subseteq V(X)\cup V(S)$. Since $X$ (resp. $S$) is  $K_{1,l}$-free (resp. $K_{l,t}$-free),  we have $|C(T)|\le l-1$, a contradiction. So $G_1^x$ is $K_{l,t}$-free.

  Note that $G_1^x[V_F]$ for any $F\in \mathscr{F}$ and $G_1^x[U]$ are empty graphs. By Lemma \ref{thm: TB theorem}, we have the matching number of $G_1^x$ is no more than $s$. We can make the similar proof for $G_2^x$.
\end{proof}

\begin{lemma}\label{con-lem-1}Let $m_1,\ldots,m_s\ge 1$ and $s\ge 2$. Then we have
$$
    \sum_{i=1}^sex(m_i,K_{l,t})\leq ex\left(\sum_{i=1}^sm_i-s+1,K_{l,t}\right).
$$
\end{lemma}
\begin{proof} We just show the case $s=2$ and then the result holds by induction.
Since $ex(1,K_{l,t})=0$, the result holds if $m_1=1$ or $m_2=1$. So we assume $m_1,m_2\geq 2$. Let $G_1$ and $G_2$ be two $K_{l,t}$-free graphs on $m_1$ and $m_2$ vertices respectively such that $e(G_i)=ex(m_i,K_{l,t})$ for $i=1,2$ and $V(G_1)\cap V(G_2)=\emptyset$. Choose $u_1\in V(G_1)$ and $u_2\in V(G_2)$. Construct a new graph $G'$ with $V(G')=(V(G_1)\cup V(G_2))\setminus\{u_2\}$ and $E(G)=(E(G_1)\cup E(G_2)\cup \{u_1v|v\in N_{G_2}(u_2)\})\setminus \{u_2v|v\in N_{G_2}(u_2)\}$. Then $G'$ is a $K_{l,t}$-free graph of order  $m_1+m_2-1$ and $e(G')= ex(m_1,K_{l,t})+ex(m_2,K_{l,t})$. Hence  the result holds.
\end{proof}

\section{Proof of Theorem \ref{thm: ex for k_l,t}}\label{sec: maximum k_lt free}

 Let $G$ be a $M_{s+1}$-free graph with $n$ vertices. By Lemma \ref{thm: TB theorem}, there is a subset $X\subseteq V(G)$ such that (\ref{eq: TB inequaulity}) holds. Denote $\mathscr{X}=\{X\subseteq V(G)~|~\mbox{X \text{satisfies (\ref{eq: TB inequaulity})}}\}$
and $x(G)=\max\{|X|~|~X\in \mathscr{X}\}.$
Then $1\leq x(G)\leq s$. Given  $X\in \mathscr{X}$ with $|X|=x(G)$, we use $C_1,\ldots, C_{m}$ to denote the components of $G-X$ and let $I=\{i~|~ |V(C_i)| \geq 2, 1\leq i\leq m\}$.
Set $$\mathscr{G}_x:=\{G~|~ \text{$G$ is $\{K_{l,t},M_{s+1}\}$-free with  $|V(G)|=n$ and $x(G)=x$}\}.$$
 Denote $ex(n,\{K_{l,t},M_{s+1}\},x)=\max_{G\in \mathscr{G}_x}e(G)$. Then we have $$ex(n,\{K_{l,t},M_{s+1}\})=\max_{1\leq x\leq s}ex(n,\{K_{l,t},M_{s+1}\},x).$$ We first give some results about $ex(n,\{K_{l,t},M_{s+1}\},x)$.
\begin{lemma}\label{lem: ex(n,klt,k)< ex(n,klt,l-1)}
    Let $n\geq 2{3s\choose 2}$ and $s\ge l+1\geq 3$. Then for all $1\leq x\leq l-1$, we have $$ex(n,\{K_{l,t},M_{s+1}\},l-1)\geq ex(n,\{K_{l,t},M_{s+1}\},x).$$
\end{lemma}
\begin{proof}
   Obviously $K_{l-1,n-(l-1)}\in \mathscr{G}_{l-1}$. Then we have that $ex(n,\{K_{l,t},M_{s+1}\},l-1)\geq (l-1)(n-l+1)$. Assume $1\le x\leq l-2$ and $G\in \mathscr{G}_x$. Let $X\in \mathscr{X}$ such that $|X|=x(G)=x$. Then $|X|+\sum_{i=1}^m \left\lfloor\frac{|V(C_i)|}{2}\right\rfloor\leq s$. Since $|I|\le s$, we have $\sum_{i\in I} |V(C_i)|\leq 3s$. Note that $e(G[X])\leq {x\choose 2}$ and $\sum_{i\in I}e(C_i)\leq {\sum_{i\in I}|V(C_i)|\choose 2}\leq {3s\choose 2}$.
   Hence
$$e(G)\leq e(G[X])+(n-x)x+\sum_{i\in I}e(C_i)\le {x\choose 2}+(n-x)x+{3s\choose 2}.$$
  So for all $1\le x\leq l-2$ and $n\geq 2{3s\choose 2}$, we have $$ex(n,\{K_{l,t},M_{s+1}\},x)\le (l-1)(n-l+1)\leq ex(n,\{K_{l,t},M_{s+1}\},l-1).$$
\end{proof}
\begin{lemma}\label{lem: calculate ex(n,klt,l-1)}
Let $n\geq 2s+1$, $2\leq l\leq t$ and $ s \geq l+1$. Then we have
$$ex(n,\{K_{l,t},M_{s+1}\},l-1)\leq  (l-1)n+ex(2(s-l+1)+1,K_{l,t})-\frac{l(l-1)}{2}.$$
\end{lemma}


\begin{proof}
   Let  $G \in\mathscr{G}_{l-1}$ such that $e(G)=ex(n,\{K_{l,t},M_{s+1}\},l-1)$. Let $X\in \mathscr{X}$ such that $|X|=l-1$.
 Since $G$ is $K_{l,t}$-free, we have $e(G[C_i])\leq ex(|V(C_i)|,K_{l,t})$. Note that $\sum_{i=1}^m e(V(C_i),X)\leq (l-1)(n-l+1).$ So
 \begin{equation*}
     \begin{aligned}
         e(G)\leq &e(G[X])+\sum_{i=1}^m  e(G[C_i])+\sum_{i=1}^m e(V(C_i),X)\\
         \leq&\frac{(l-1)(l-2)}{2}+\sum_{i=1}^m  ex(|V(C_i)|,K_{l,t})+(l-1)(n-l+1)\\
         =& (l-1)n+\sum_{i=1}^m  ex(|V(C_i)|,K_{l,t})-\frac{l(l-1)}{2}.
     \end{aligned}
 \end{equation*}
 By Lemma \ref{con-lem-1}, we have $\sum_{i=1}^m  ex(|V(C_i)|,K_{l,t})\le ex(\sum_{i=1}^m|V(C_i)|-m+1,K_{l,t})$.
We left to prove $\sum_{i=1}^m|V(C_i)|\leq 2(s-l+1)+m$.
We may assume $|V(C_i)|$ are all odd, i.e. $|V(C_i)|=2k_i+1$ ($1\leq i\leq m$). By Lemma \ref{thm: TB theorem}, we have $\sum_{i=1}^mk_i \le s-l+1$ and then $\sum_{i=1}^m|V(C_i)| \le 2(s-l+1)+m$. Hence we have $\sum_{i=1}^m  ex(|V(C_i)|,K_{l,t})\leq ex(2(s-l+1)+1,K_{l,t})$ and we are done.
\end{proof}

 Combining Lemmas \ref{lem: ex(n,klt,k)< ex(n,klt,l-1)} and \ref{lem: calculate ex(n,klt,l-1)}, we have the following result.

\begin{cor}\label{new2} Let $n\geq 2{3s\choose 2}$ and $s\ge l+1\geq 3$. For all $1\le x\leq l-1$, we have that
$$
    ex(n,\{K_{l,t},M_{s+1}\},x)\leq  (l-1)n+ex(2(s-l+1)+1,K_{l,t})-\frac{l(l-1)}{2}.
$$
\end{cor}

 For $l\le x\leq s$, we have the following results for the upper bound.
\begin{lemma}\label{lem: ex(n,klt >l) leq xxx} Let $n\geq 2s+1$, $t\ge l$ and $2\le l\le x\le s$. Then we have $$
    ex(n,\{K_{l,t},M_{s+1}\},x)\leq (t-1){x\choose l}+(l-1)n-\left\lceil\frac{x(l-1)}{2}\right\rceil+ex(2(s-x)+1,K_{l,t}).
$$
\end{lemma}
\begin{proof} Choose $G \in\mathscr{G}_{x}$ such that $e(G)=ex(n,\{K_{l,t},M_{s+1}\},x)$. Let $X\in \mathscr{X}$ such that $|X|=x$. Suppose there is $v\in V(G)$ such that $d_X(v)\le l-2$. Since $|X|=x\ge l$, there is $u\in X\setminus N_X(v)$ such that $uv\notin E(G)$. Let $G'=G+uv$. By Lemma \ref{thm: TB theorem} and $d_X(v)\le l-2$, we have $G'\in \mathscr{G}_{k}$ with $e(G')>e(G)$, a contradiction with the choice of $G$. Hence
 $d_X(v)\ge l-1$ for any $v\in V(G)$.

Let $T$ be a  $l$-horn of $X$. Since $G$ is $K_{l,t}$-free,  $|C(T)|\le t-1$. Thus by calculating the number of $l$-horns of $X$, we have
 \begin{equation*}\label{eq: calculate horns}
     \begin{aligned}
         (t-1){x\choose l}&\geq \sum_{v\in V(G)}{d_X(v)\choose l}=\sum_{v\in X}{d_X(v)\choose l}+\sum_{v\notin X}{d_X(v)\choose l}\\
         &\geq \sum_{v\in X}(d_X(v)-(l-1))+\sum_{v\notin X}(d_X(v)-(l-1))\\
     &\geq \frac{1}{2}\sum_{v\in X}d_X(v)-\left\lfloor\frac{x(l-1)}{2}\right\rfloor+\sum_{v\notin X}(d_X(v)-(l-1))\\
                 &= \frac{1}{2}\sum_{v\in X}d_X(v)+\sum_{v\notin X}d_X(v)-(l-1)n+\lc\frac{(l-1)x}{2}\rc.
    \end{aligned}
\end{equation*}
 Since $e(G)=\frac{1}{2}\sum_{v\in X}d_X(v)+\sum_{v\notin X}d_X(v)+\sum_{i=1}^{m}e(C_i)$ and $e(C_i)\leq ex(|V(C_i)|,K_{l,t})$, we have
$$e(G)\leq (t-1){x\choose l}+(l-1)n-\left\lceil\frac{(l-1)x}{2}\right\rceil+\sum_{i=1}^m ex(|V(C_i)|,K_{l,t}),$$
where
\begin{equation}\label{eq: consriction of TB thm}
    \sum_{i=1}^m\left\lfloor\frac{|V(C_i)|}{2}\right\rfloor+k\leq s.
\end{equation}
Similar to the proof of Lemma \ref{lem: calculate ex(n,klt,l-1)}, we have that $\sum_{i=1}^m ex(|V(C_i)|,K_{l,t})\leq ex(2(s-x)+1,K_{l,t})$. Then
$$
    e(G)\leq (t-1){x\choose l}+(l-1)n-\left\lceil\frac{x(l-1)}{2}\right\rceil+ex(2(s-x)+1,K_{l,t})
$$and we are done.
\end{proof}

Note that $ex(2(s-x)+1,K_{l,t})=O((s-x)^{2-1/t}).$ If  $l\ge 3$, then the following function
$$(t-1){x\choose l}+(l-1)n-\left\lceil\frac{x(l-1)}{2}\right\rceil+ex(2(s-x)+1,K_{l,t})$$
 is monotonically increasing when $s$ is large  enough. Thus by Lemma \ref{lem: ex(n,klt >l) leq xxx}, we easily have the following corollary.
\begin{cor}\label{new}
Let $n\geq 2s+1$, $t\ge l$ and $2\le l\le x\le s$. If $s$ is large  enough, then we have $$ex(n,\{K_{l,t},M_{s+1}\},x)\leq (l-1)n+(t-1){s\choose l}-\lc\frac{s(l-1)}{2}\rc.$$
\end{cor}

Now we are going to prove Theorem \ref{thm: ex for k_l,t}.
\vskip.2cm

\begin{reproof} The lower bound in  Theorem \ref{thm: ex for k_l,t} is given by the Constructions \ref{con: k_l,t-free 1} and \ref{con: k_l,t-free 2}. So we just show the upper bound.


Since $ex(n,\{K_{l,t},M_{s+1}\})=\max_{1\leq x\leq s}ex(n,\{K_{l,t},M_{s+1}\},x)$, by Corollaries \ref{new2} and \ref{new}, we have $$\begin{array}{rcl}
& &ex(n,\{K_{l,t},M_{s+1}\})\\
&\leq &\max\left\{(l-1)n+ex(2(s-l+1)+1,K_{l,t})-\frac{l(l-1)}{2},(l-1)n+(t-1){s\choose l}-\lc\frac{s(l-1)}{2}\rc\right\}.\end{array}$$Note that $ex(2(s-l+1)+1,K_{l,t})=O((s-x)^{2-1/t})$.
When $l\geq 2$ and $s$ is large enough, $(l-1)n+(t-1){s\choose l}-\lc\frac{s(l-1)}{2}\rc>(l-1)n+ex(2(s-l+1)+1,K_{l,t})-\frac{l(l-1)}{2}$. So
$$ex(n,\{K_{l,t},M_{s+1}\})\leq (l-1)n+(t-1){s\choose l}-\lc\frac{s(l-1)}{2}\rc.$$
\end{reproof}

If we replace the condition that $s$ is large enough  with $s\ge l+1$, we obtain a weaker conclusion.
\begin{cor}
    Let $3\leq l\leq t$, $s\geq l+1$ and $n\geq 2{3s\choose 2}$. Then we have $$\exkm\leq (t-1){s\choose l}+(l-1)n-\frac{l(l-1)}{2}+ex(2(s-l)+1,K_{l,t}).$$
\end{cor}
\begin{proof}
  Since $ex(n,\{K_{l,t},M_{s+1}\})=\max_{1\leq x\leq s}ex(n,\{K_{l,t},M_{s+1}\},x)$, by Corollaries \ref{new2} and \ref{new}, we just need  to prove that when $1\le x\leq l-1$, $ex(n,\{K_{l,t},M_{s+1}\},x)\leq (t-1){s\choose l}+(l-1)n-\frac{l(l-1)}{2}+ex(2(s-l)+1,K_{l,t}).$
Equivalent to proof $$(l-1)n+ex(2(s-l+1)+1,K_{l,t})-\frac{l(l-1)}{2}\leq (t-1){s\choose l}+(l-1)n-\frac{l(l-1)}{2}+ex(2(s-l)+1,K_{l,t}),$$
    i.e. to proof
    $$ex(2(s-l+1)+1,K_{l,t})\leq (t-1){s\choose l}+ ex(2(s-l)+1,K_{l,t}).$$
    Set $a=s-l$. Then $a\geq 1$. Since $t\geq l\geq 3$ and $ex(2a+3,K_{l,t})-ex(2a+1,K_{l,t})\leq 4a+3$, we have
    \begin{equation*}
        \begin{aligned}
             (t-1){s\choose l}\geq& 2{l+a\choose l}\geq 2{a+3\choose 3}=\frac{(a+3)(a+2)(a+1)}{3}\\
             \geq& 4a+3\geq ex(2a+3,K_{l,t})-ex(2a+1,K_{l,t})
        \end{aligned}
    \end{equation*}
    and we are done.
\end{proof}
 When $t\geq 2(s-l)+2$, we can determine the exact value of $\exkm$ if $n\geq 2{3s\choose 2}.$ Recall
 $$F_1(x)=(l-1)n+(t-1){x\choose l}+{2(s-x)+1\choose 2}-\lc\frac{x(l-1)}{2}\rc.$$
\begin{theorem}
    Let $2l+2\leq 2s\leq 2l+t-2$, $l\geq 3$, $t\ge 4$ and $n\geq \max\{2{3s\choose 2},n(s)\}$.
    Then we have
    $$\exkm= \max\{F_1(l),F_1(s)\}.$$
\end{theorem}
\begin{proof}
Since $2s\leq 2l+t-2$, we have $2s-t+1\leq 2l$. For $t\geq l\geq 3$, we have $n_1 = n(s)\leq n$. By Theorem \ref{thm: ex for k_l,t}, we have
\begin{equation}\label{low}
\exkm\geq \max_{l\leq x\leq s }\{F_1(x))\}.\end{equation}
By Lemma \ref{lem: ex(n,klt >l) leq xxx}, for any $l\leq x\leq s$,
$$ex(n,\{K_{l,t},M_{s+1}\},x)\leq (l-1)n+(t-1){x\choose l}+ex(2(s-x)+1,K_{l,t})-\lc\frac{x(l-1)}{2}\rc=F_1(x).$$
By Corollary \ref{new2}, for any $1\le x\le l-1$,
 $$ex(n,\{K_{l,t},M_{s+1}\},x)\leq (l-1)n+ex(2(s-l+1)+1,K_{l,t})-\frac{l(l-1)}{2}.$$ Note that $ex(2(s-l+1)+1,K_{l,t})\leq ex(2(s-l-1)+1,K_{l,t})+8(s-l)+2.$ Since $2(s-l-1)+1<t$, $ex(2(s-l-1)+1,K_{l,t})={2(s-l-1)+1\choose 2}$. Then for any $1\le x\le l-1$, we have
$$\begin{array}{rcl}
& &ex(n,\{K_{l,t},M_{s+1}\},x)\leq(l-1)n+ex(2(s-l+1)+1,K_{l,t})-\frac{l(l-1)}{2}\\
        &\le& (l-1)n+ex(2(s-l-1)+1,K_{l,t})+8(s-l)+2-\frac{l(l-1)}{2}\\
        &\le &(l-1)n+ex(2(s-l-1)+1,K_{l,t})+4(t-\frac{3}{2})+\frac{l+1}{2}-\lfloor\frac{l-1}{2}\rfloor-1-\frac{l(l-1)}{2}\\
        &\le &(l-1)n+ex(2(s-l-1)+1,K_{l,t})+(t-1)(l+1)-\lc\frac{l-1}{2}\rc-1-\frac{l(l-1)}{2}\\
        &\leq& (l-1)n +(t-1)(l+1)
         +ex(2(s-l-1)+1,K_{l,t})-\lc\frac{(l+1)(l-1)}{2}\rc\\
         &=&F_1(l+1)\le \max_{l\leq x\leq s }\{F_1(x)\}.
     \end{array}$$
By $ex(n,\{K_{l,t},M_{s+1}\})=\max_{1\leq x\leq s}ex(n,\{K_{l,t},M_{s+1}\},x)$ and (\ref{low}), we have $$\exkm= \max_{l\leq x\leq s }\{F_1(x)\}.$$

 It is easy to check that $F_1(y+2)-F_1(y+1)\geq F_1(y+1)-F_1(y)$. So $\max_{l\leq x\leq s }\{F_1(x)\}=\{F_1(l),F_1(s)\}$ and we are done.
\end{proof}

Notice that when $2l\leq 2x\leq 2s-t$, the gap between the lower bound given in Theorem \ref{thm: ex for k_l,t} and the upper bound given in Lemma \ref{lem: ex(n,klt >l) leq xxx} is $(2(s-x)+1)(l-1)$. The reason for this gap may be the edges between $S$ and $X$ in the Construction \ref{con: k_l,t-free 2}, and this might be difficult to determine which depending on the specific value of $l,t$ and $s$.


\section{Proofs of Theorems \ref{thm: ex for k_2,2 free} and \ref{thm: ex for k_2,t free}} \label{sec: ex for k2t free}
In this section, we will discuss  $ex(n,K_{2,t})$ for $t\ge 2$. We also consider the case when $n\geq 2s+1$.

First we  give the proof of Theorem \ref{thm: ex for k_2,2 free}. Recall $ex(n,C_4)\le \frac n4(1+\sqrt{4n-3})$.

\begin{reproof1_4}
    Note that when $s\geq 12$, $n(0) = 2s+1\leq {s\choose 2}+s+1=n(s)$,
    then the Construction \ref{con: k_l,t-free 1} with $l=t=2$ and $x=s$ gives a lower bound of $ex(n,\{K_{2,2},M_{s+1}\})$, which implies when $n\geq {s\choose 2}+s+1$,
    \begin{equation}
        ex(n,\{K_{2,2},M_{s+1}\})\geq n+{s \choose 2}-\lc \frac{s}{2}\rc.
    \end{equation}
Lemma \ref{lem: calculate ex(n,klt,l-1)} implies that
$ex(n,\{K_{2,2},M_{s+1}\},1)\leq n+ex(2s-1, K_{2,2})-1.$ By Lemma \ref{lem: ex(n,klt >l) leq xxx}, for any $2\leq x\leq s$, we have
\begin{equation*}
\begin{aligned}
    ex(n,\{K_{2,2},M_{s+1}\},x)\leq& {x\choose 2}+n-\lc\frac{x}{2}\rc+ex(2(s-x)+1,K_{2,2})\\
    \leq & {x\choose 2}+n-\lc\frac{x}{2}\rc+\frac{2(s-x)+1}{4}(1+\sqrt{8(s-x)+1}).
\end{aligned}
\end{equation*}
Setting $f(x)=\frac{x(x-1)}{2}+n-\lc\frac{x}{2}\rc+\frac{2(s-x)+1}{4}(1+\sqrt{8(s-x)+1})$. We can easily find its convexity by taking the derivative. Then we have $ex(n,\{K_{2,2},M_{s+1}\},x)\leq \max\{f(s),f(2)\}$, i.e.
$$ex(n,\{K_{2,2},M_{s+1}\},x)\leq \max\left\{n+{s\choose 2}-\lc \frac{s}{2}\rc+ \frac{3}{4},n+\frac{2s-3}{4}(1+\sqrt{8s-15})\right\},$$where $2\le x\le s$.
Since $ex(n,\{K_{2,2},M_{s+1}\})=\max_{1\leq x\leq s}ex(n,\{K_{2,2},M_{s+1}\},x)$, we have
$$\begin{array}{rcl}
& &ex(n,\{K_{2,2},M_{s+1}\})\\
&\le & \max\left\{n+ex(2s-1, K_{2,2})-1,n+{s\choose 2}-\lc \frac{s}{2}\rc+ \frac{3}{4},n+\frac{2s-3}{4}(1+\sqrt{8s-15})\right\}.
\end{array}$$
When $s\geq 12$, we have $${s\choose 2}-\lc\frac{s}{2}\rc+ \frac{3}{4}\geq \frac{2s-3}{4}(1+\sqrt{8s-15})$$ and $${s\choose 2}-\lc\frac{s}{2}\rc+ \frac{3}{4}\geq \frac{2s-1}{4}(1+\sqrt{8s-7})-1\geq ex(2s-1, K_{2,2})-1.$$
So we have $ex(n,\{K_{2,2},M_{s+1}\})\leq n+{s\choose 2}-\lc\frac{s}{2}\rc$ when $s\geq 12$. Thus the result holds.
\end{reproof1_4}
Next, we consider $ex(n,K_{2,t})$ with $t\geq 3$. Recall that the upper bound of $ex(n,K_{2,t})$ mentioned before is
$(1+o(1))\sqrt{t-1}n^{3/2}/2.$ In fact, we have a more strong result.
\begin{lemma}\label{horn} Let $t\ge 3$ and $n\ge 5$. We have
$$ex(n,K_{2,t})\leq\frac{n(1+\sqrt{4(t-1)n-(4t-5)})}4.$$
\end{lemma}
\begin{proof} Let $G$ be a $K_{2,t}$-free graph with $e(G)=ex(n,K_{2,t})$ and $T$ a $2$-horn of $V(G)$. By the choice of $G$, we have $d(v)\ge 2$ for any $v\in V(G)$.
  Since $G$ is $K_{2,t}$-free, $|C(T)|\le t-1$. Hence we have
  $$       (t-1){n\choose 2}\geq \sum_{v\in V(G)}{d(v)\choose 2}=\frac{1}{2}\left(\frac{e^2(G)}{n}-2e(G)\right),$$ which implies $4e^2(G)-2ne(G)-(t-1)(n^3-n^2)\le 0$.
 Thus the result holds.
\end{proof}

\begin{reproof1_5}
    When $l=2$, $t\geq 3$, $n(0) = 2s+1\leq {s\choose 2}+s+1=n(s)$, then the Construction \ref{con: k_l,t-free 1} with $l=2$ and $x=s$ gives a lower bound of $ex(n,\{K_{2,t},M_{s+1}\})$, which implies when $n\geq {s\choose 2}+s+1$,
    \begin{equation}\label{low-bound} ex(n,\{K_{2,t},M_{s+1}\})\geq n+(t-1){s\choose l}-\lc\frac{s}{2}\rc.\end{equation}
By Lemma \ref{lem: calculate ex(n,klt,l-1)}, we have $ex(n,\{K_{2,t},M_{s+1}\},1)\leq n+ex(2s-1,K_{2,t})-1$. By Lemmas \ref{lem: ex(n,klt >l) leq xxx} and \ref{horn}, for  $2\leq x\leq s$, we have
\begin{equation*}
\begin{aligned}
    &ex(n,\{K_{2,t},M_{s+1}\},x)\leq (t-1){x\choose 2}+n-\lc\frac{x}{2}\rc+ex(2(s-x)+1,K_{2,t})\\
    \leq &(t-1) {x\choose 2}+n-\lc\frac{x}{2}\rc
    +\frac{2(s-x)+1}{4}(1+\sqrt{4(t-1)(2(s-x)+1)-(4t-5)}).
\end{aligned}
\end{equation*}
Setting $$h(x) = (t-1){x\choose 2}+n-\lc\frac{x}{2}\rc+\frac{2(s-x)+1}{4}(1+\sqrt{4(t-1)(2(s-x)+1)-(4t-5)}),$$
we can easily find its convexity by taking the derivative. Then we have $h(x)\le\max\{h(2),h(s)\}$, i.e.,
\begin{equation}\label{6}
\begin{aligned}
 &ex(n,\{K_{2,t},M_{s+1}\},x)\\
 \leq & \max\left\{n+t-2+\frac{2s-3}{4}(1+\sqrt{4(t-1)(2s-3)-4t+5}),n+(t-1){s\choose 2}-\lc\frac{s}{2}\rc+\frac{3}{4}\right\},
\end{aligned}
\end{equation}where $2\le x\le s$.

When $s$
is large enough, we have that
$$n+(t-1){s\choose 2}-\lc\frac{s}{2}\rc+\frac{3}{4}\geq n+t-2+\frac{2s-3}{4}(1+\sqrt{4(t-1)(2s-3)-4t+5})$$
and
$$n+(t-1){s\choose 2}-\lc\frac{s}{2}\rc+\frac{3}{4}\geq n+ex(2s-1,K_{2,t})-1.$$
Since $ex(n,\{K_{2,2},M_{s+1}\})=\max_{1\leq x\leq s}ex(n,\{K_{2,2},M_{s+1}\},x)$,  when $s$ is large enough, we have $$ex(n,\{K_{2,t},M_{s+1}\})\leq n+(t-1){s\choose 2}-\lc\frac{s}{2}\rc+\frac{3}{4}.$$ Since $ex(n,\{K_{2,t},M_{s+1}\})$ is an integer,  $ex(n,\{K_{2,t},M_{s+1}\})\leq n+(t-1){s\choose 2}-\lc\frac{s}{2}\rc$ and we are done.\end{reproof1_5}

\vskip.2cm
\noindent{\bf Remark} In the proof of Theorem \ref{thm: ex for k_2,t free}, we have $h(2)\le h(s)$ and $ex(n,\{K_{2,t},M_{s+1}\},1)\leq n+ex(2s-1,K_{2,t})-1\le h(1).$ According to the convexity of $h$ and (\ref{6}), we have
    $$ex(n,\{K_{2,t},M_{s+1})\})\leq \max\{h(1),h(s)\}.$$
    Setting $s=y+1$. Then
    \begin{equation*}
        \begin{aligned}
            h(s)-h(1) = (t-1)\frac{(y+1)y}{2}-\frac{y+1}{2}+\frac{3}{4}
            -\frac{3(y+1)}{4}\sqrt{4(t-1)(2y+2-1)-4t+5}+1.
        \end{aligned}
    \end{equation*}
    When $t\geq 6$, for $y\geq 0$, we have
    \begin{equation*}
        \begin{aligned}
            &\frac{3(y+1)}{4}\sqrt{4(t-1)(2y+1)-4t+5}-1\\
            =& \frac{3(y+1)}{4}\sqrt{8y(t-1)+1}-1\\
            \leq &\frac{(t-1)(y^2+y)}{2}\\
            \leq &\frac{(t-1)}{2}y^2+\frac{t-2}{2}y\\
            \leq &\frac{(t-1)(y^2+y)}{2}-\frac{y+1}{2}+\frac{3}{4}.
        \end{aligned}
    \end{equation*}
     So when $t\geq 6$, $h(s)\geq h(1)$ for every $s\geq 2$. Combining with (\ref{low-bound}),
      we have the following corollary.

\begin{cor}
    Let $n\geq {s\choose 2}+s+1$, $t\geq 6$ and $s\geq 3$. Then we have $$ex(n,\{K_{2,t},M_{s+1}\})=n+(t-1){s\choose 2}-\lc\frac{s}{2}\rc.$$
\end{cor}

\section{Conclusion}
In this paper, we give the results about maximum number of edges of $K_{l,t}$-free graphs with bounded matching number. In the case of $l\ge3$, when $s$ is sufficiently large with respect to $l$ or when $t$ is sufficiently large with respect to $s-l$, we can give the exact value of $\exkm$ with the extremal graph. But in other case, there exists a gap between the lower bound and the upper bound of $\exkm$. In the case of $l=2$, we give more specific result of $\excm$ and $ex(n,\{K_{2,t},M_{s+1}\})$ when $s$ is larger than some values depending on $t$. From the proofs of our main Theorems, we guess the exact value of $\exkm$ will be attached at the  Construction \ref{con: k_l,t-free 2} by adding some edges between $S$ and $X$ when $s$ is relatively small.
But it is still difficult to get the specific value when $s$ is relatively small. We believe it depends on the specific value of $s,l$ and $t$.

\section*{Acknowledgement}
	
This research was supported by the National Natural Science Foundation of China (Grant 12171272 \& 12161141003).

\end{document}